\documentclass[10pt]{amsart}

\usepackage[utf8]{inputenc}
\usepackage[T1]{fontenc}
\usepackage[pdftex]{graphicx}
\usepackage{latexsym}
\usepackage{amsmath,amssymb,amsthm}
\usepackage{mathtools}
\usepackage{enumitem}
\usepackage{faktor}
\usepackage{enumerate}
\usepackage{hyperref}
\usepackage{tikz}
\usepackage{xcolor}

\newtheorem{theorem}{Theorem}[section]
\newtheorem{corollary}[theorem]{Corollary}

\newtheorem{lemma}[theorem]{Lemma}
\theoremstyle{definition}

\newtheorem*{remark}{Remark}

\newcommand{\Q}{\mathcal{Q}}
\newcommand{\lk}{\mathrm{lk}(\rho)}

\newcommand{\Hmm}[1]{\leavevmode{\marginpar{\tiny%
			$\hbox to 0mm{\hspace*{-0.5mm}$\leftarrow$\hss}%
			\vcenter{\vrule depth 0.1mm height 0.1mm width \the\marginparwidth}%
			\hbox to 0mm{\hss$\rightarrow$\hspace*{-0.5mm}}$\\\relax\raggedright #1}}}
	
	\newcommand{\eat}[1]{}

\begin{document}
	\author[P. Bartmann, M. Keller]{Philipp Bartmann, Matthias Keller}
	
	\address{Philipp Bartmann: Institut f\"ur Mathematik, Universit\"at Potsdam
		14476  Potsdam, Germany}
	\email{philipp.bartmann@uni-potsdam.de}
    \address{Matthias Keller:   Israel Institute of Advanced Studies, Jerusalem, Israel; Institut f\"ur Mathematik, Universit\"at Potsdam
14476  Potsdam, Germany}
	\email{matthias.keller@uni-potsdam.de}

	\title[The complex property of the boundary operator]{The complex property of the boundary operator  on simplicial complexes}

\begin{abstract}
We study the complex property $\partial\partial = 0$ of the boundary operator $\partial$ on a weighted, infinite, and possibly non-locally finite simplicial complex. We give a characterization of this property in $\ell^2$ in terms of the recurrence of the links of simplices.
The complex property is essential to ensure that Hodge Laplacians $\Delta^H $ indeed act as  $\delta\partial + \partial\delta$ and to decompose $\Delta^H$ into a direct sum of operators acting on $k$-forms.  Furthermore, it allows us to define relative cohomology classes, show a respective weak Hodge decomposition, and prove the existence of harmonic Dirichlet eigenforms. We also discuss a transience property for simplicial complexes, that was  introduced by Parzanchevski and Rosenthal.
\end{abstract}
\maketitle

\section{Introduction}
The boundary and coboundary operators $\partial$ and $\delta$ are fundamental objects in the theory of simplicial complexes. The axiomatic \emph{complex property} of the coboundary operator says $$\delta\delta = 0.$$ 
The boundary operator $\partial$ is defined as the formal adjoint of $\delta$ and therefore it is natural to ask if $\partial\partial = 0$ holds as well. This is clearly true for finite simplicial complexes, but may fail in the infinite setting. Indeed, it was already observed by Parzanchevski and Rosenthal  in \cite[Section~3.1]{Parzanchevski} that there may be $1$-forms $\omega$ such that $$\partial\partial \omega (\varnothing)\neq0$$ whenever the set of vertices  is infinite. They also mention that such an $\omega$ can even be chosen in $\ell^2$ when the $1$-skeleton gives rise to a transient graph such as $3$-dimensional Euclidean lattices or regular trees.

This phenomenon was further explored in \cite[Example~2.5]{BK}, where it was shown that $\partial\partial\ne0$ can also occur for forms of higher degree whenever the simplicial complex is not locally finite. Clearly, to give meaning to the objects in non-locally finite settings, one has to consider weights which satisfy a local summability condition similar to the case of graphs, \cite{Soardi,KLW}. Strengthening the local summability condition, it was shown in \cite[Lemma~3.2]{BK} that $\partial\partial = 0$ holds on $\ell^2.$

In this note, we explore this phenomenon systematically and  characterize the complex property of the boundary operator $\partial$, i.e., when $$\partial\partial = 0$$ holds on $\ell^2$. We give a  characterization 
in terms of the recurrence of the graphs arising from links of simplices. In this sense, we draw a surprising connection between a probabilistic property of links and a topological property of the simplicial complex. 
The underlying idea is a localization technique in spirit of Garlands method \cite{Garland}, which allows to relate the boundary and coboundary operator to the graph Laplacians on links.

Subsequently, we give various applications of this characterization. A first motivation is the study of the Hodge Laplacian $\Delta^H,$ which is a fundamental object in the theory of simplicial complexes and its spectral properties are of great interest.
Our aim is to guarantee that  $\Delta^H $ acts indeed as
\[\Delta^H =   \delta\partial + \partial\delta\]
on its particular domain. While in general one only knows that $\Delta^H$ acts as $(\delta+\partial)^2$, the above factorization allows to decompose $\Delta^H$ into a direct sum of Hodge Laplacians $\Delta^H_k$ acting on $k$-forms. Clearly, $\partial\partial\neq 0$ is an obstruction to this decomposition as $k$-forms do not get mapped to $k$-forms again.

As a second application, we turn to the topic of Hilbert complexes with ideal boundary conditions as introduced in \cite{BL} for Riemannian manifolds. For simplicial complexes,  it was shown in \cite{BK} that $\delta,$ restricted to the  $\ell^2$ forms $\omega$ such that $\delta\omega\in\ell^2,$ always gives rise to a Hilbert complex with so called absolute boundary conditions. However, it was left open whether the same is true for $\partial.$ We show that this is indeed the case if all links are recurrent. This then allows to define the corresponding relative cohomology classes and to show that the weak Hodge decomposition holds in this setting. We also show that given $\partial\partial = 0$, the up-Laplacian, acting as $\partial\delta,$ has  a non-empty kernel consisting of eigenforms. These ``trivial'' eigenforms obviously exist for finite simplicial complexes, and this property now extends to infinite simplicial complexes with recurrent links. 

Finally, we draw a connection to a transience property $(\mathrm T')$ proposed by Parzan\-chevski and Rosenthal in \cite{Parzanchevski}.

The paper is structured as follows. In Section~2 we introduce the setting of Laplacians on simplicial complexes and graphs. In Section~3 we explain the localization technique. In Section~4 we give the characterization of $\partial\partial = 0$ in terms of the recurrence of links. In Section~5 we discuss the applications mentioned above.

\section{Laplacians on  complexes and graphs}
In this section, we introduce the setting of Laplacians on simplicial complexes and graphs. We follow the setting from \cite{BK} for simplicial complexes and the setting from \cite{KLW} for graphs. We also explain how graphs are included in the setting of simplicial complexes. However, since we are interested in associating link graphs of simplicial complexes, it is convenient to have a distinct notation for graphs at hand.\\
For a discrete set $X$, we denote the space of complex functions on $X$ by $C(X)$ and refer to the subspace of functions of finite support as $C_c(X).$ We also fix the notation \[\sum_{x\in X}f(x) = \sum_X f\] for $f\in C(X),$ whenever $f$ is nonnegative or the sum is absolutely summable.
\subsection{Laplacians on simplicial complexes}
Given a discrete set $X$ of \textit{vertices}, we call a set $\hat\Sigma\subseteq\mathcal{P}(X)$ of finite subsets of $X$ a \textit{simplicial complex} if it is closed under the operation of taking subsets. That means, if $\sigma\in\hat\Sigma$ and $\tau\subseteq\sigma,$ then $\tau\in\hat\Sigma.$ We call the elements of $\hat\Sigma$ \textit{simplices}. We refer to simplices of cardinality $k+1$ as \emph{$k$-simplices} and the set of $k$-simplices as $\hat\Sigma_k.$ The \emph{dimension} of a simplex $\tau$ is given by $\dim(\tau) = \vert\tau\vert-1$ and the dimension of $\hat\Sigma$ is $\dim(\hat\Sigma) = \sup_{\tau\in\hat\Sigma}\dim(\tau).$ Whenever it is convenient we will identify $\hat\Sigma_0$ with $X.$ Whenever $\sigma\in\hat\Sigma$ and $\tau\subseteq\sigma$ with $\vert\sigma\backslash\tau\vert = 1$, we write $\tau\prec\sigma$ or $\sigma\succ\tau$ and say that $\tau$ is a \textit{face} of $\sigma$ or that $\sigma$ is a \textit{coface} of $\tau.$\\
We fix a weight function $m:\hat\Sigma\rightarrow(0,\infty),$ which we interpret as a discrete measure on $\hat\Sigma$ by the convention $m(A) = \sum_{\tau\in A}m(\tau)$ for all $A\subseteq\hat\Sigma.$ \\
In order to define operators and function spaces for infinite simplicial complexes that are natural and common in the finite setting, we will need that 
\[\sum_{\sigma\succ\tau}m(\sigma)<\infty\] 
for all $\tau .$ Notice, however, that every simplicial complex contains the single $-1$-dimensional simplex $\varnothing,$ for which the above requirement means \[\sum_{x\in X}m(x)<\infty\] because every vertex $x\in X$ is a coface of $\varnothing.$ This is, of course, very restrictive since, on the level of graphs, this means that the total measure is finite. 
Simply excluding $\varnothing$ from that requirement in all cases is not a good option either, as it breaks consistency with the established theory for finite simplicial complexes. Indeed, in some cases it is desirable to have $\varnothing$ included. 
To stay flexible and not exclude $\varnothing$ if not necessary, we define $\Sigma$ as $\Sigma = \hat\Sigma$ if $\sum_{x\in X}m(x)<\infty$ and $\Sigma = \hat\Sigma\backslash\lbrace\varnothing\rbrace$ otherwise and demand \[\sum_{\sigma\succ\tau}m(\sigma)<\infty\] for all $\tau\in\Sigma.$ In that case we call $m$ \textit{locally summable} and call the pair $(\Sigma,m)$ a \textit{weighted simplicial complex}, which is a standing assumption for the rest of the paper.\\
Now let $1_\tau\in C_c(\Sigma)$ be the characteristic function of $\tau\in\Sigma.$
We say that a linear operator $\delta:C(\Sigma)\rightarrow C(\Sigma)$ is a \textit{coboundary} operator on $\Sigma$ if it satisfies $\delta 1_\tau(\sigma)\in\lbrace\pm1,0\rbrace$ for all $\tau,\sigma\in\Sigma $,  $\delta1_\tau(\sigma)\neq 0$ iff $\tau\prec\sigma$ and
\[\delta\delta = 0.\]

 Since every simplex has finitely many faces, the action of $\delta$ is given for every $\omega\in C(\Sigma)$ and $\sigma\in\Sigma$ by the finite sum \[\delta\omega(\sigma) = \sum_{\tau\prec\sigma}\theta(\tau,\sigma)\omega(\tau),\] where $\theta(\tau,\sigma) = \delta1_\tau(\sigma)\in\lbrace\pm1\rbrace.$

\begin{remark}
    We  define the coboundary operator on functions over $\Sigma,$ rather than, more commonly, alternating forms over oriented simplices, cf. \cite{HJ,Parzanchevski,RT}. We do this because it naturally fits the formalism of weighted graphs from \cite{KLW} and eliminates the need to define oriented simplicial complexes. Both settings are essentially equivalent, up to a choice of \textit{orientation} that is, in essence, an ordering of the set of vertices $X$, see \cite[Subsection~2.2]{BK}.
\end{remark} 
 
We define the \textit{boundary operator} $\partial: D(\partial)\rightarrow C(\Sigma)$ through the action
\[\partial\omega(\rho) = \frac{1}{m(\rho)}\sum_{\tau\succ\rho}m(\tau)\theta(\rho,\tau)\omega(\tau)\] for all $\rho\in\Sigma$ and $\omega\in D(\partial)$ with \[D(\partial) = \lbrace\omega\in C(\Sigma)\mid\sum_{\tau\succ\rho}m(\tau)\vert\omega(\tau)\vert<\infty\quad\text{for all}~\rho\in\Sigma\rbrace.\]
With the boundary and coboundary operator at hand we define the quadratic forms
that give rise to the up-, down-, and Hodge Laplacian
\begin{align*}
  \Q^+:C(\Sigma)\rightarrow[0,\infty],&&  \Q^+(\omega) &= \sum_{\Sigma}m\vert\delta\omega\vert^2,\\
  \Q^-:D(\partial)\rightarrow[0,\infty],&&  \Q^-(\omega) &=\sum_{\Sigma}m\vert\partial\omega\vert^2,\\
 \Q^H:D(\partial)\rightarrow[0,\infty],&&   \Q^H(\omega) &= \sum_{\Sigma}m\vert\delta\omega+\partial\omega\vert^2.
\end{align*}
The Banach spaces $\ell^p = \ell^p(\Sigma,m)$, $p\in[1,\infty]$, with norm ${\|\cdot\|}_p$ are defined in the usual way, and we note  that
$\ell^2$ is a Hilbert space with inner product $\langle f,g\rangle = \sum_\Sigma mf\overline{g}.$ Notice that by local summability and the H\"older inequality, one always has $\ell^p\in D(\partial)$ for all $p\in [1,\infty].$\\
Let $\circ\in\lbrace \pm,H\rbrace.$ We are interested in two particular restrictions of $\Q^\circ,$ which arise from natural boundary conditions. Let $Q^\circ_D\subseteq\Q^\circ$ be the closed quadratic form on $\ell^2$ with domain \[D(Q^\circ_D) =  \lbrace \omega\in\ell^2\mid\lim_{n\to\infty}\left(\Vert\omega_n-\omega\Vert^2_2+\mathcal{Q}^\circ(\omega_n-\omega)\right) = 0 \mbox{ for some }\omega_n\in C_c(\Sigma)\,\rbrace.\]
Moreover, let $Q^\circ_N\subseteq\Q^\circ$ be the closed quadratic form with domain 
\[D(Q^\circ_N) = \lbrace\omega\in\ell^2\mid\Q^\circ(\omega)<\infty\rbrace.\]

We also define $$\partial_D = \partial\mid_{D(Q^-_D)},\quad \partial_N = \partial\mid_{D(Q^-_N)},\quad\delta_D = \delta\mid_{D(Q^+_D)}\quad\mbox{ and }\quad\delta_N = \delta\mid_{D(Q^+_N)}$$ 
as well as $(\delta+\partial)_D = (\delta+\partial)\mid_{D(Q^H_D)}$ and $(\delta+\partial)_N = (\delta+\partial)\mid_{D(Q^H_N)}.$
We call $\delta_D$ and $\partial_D$ (respectively  $\delta_N$ and $\partial_N$) the coboundary and boundary operator with \textit{Dirichlet} (respectively \textit{Neumann}) \textit{boundary conditions}. These operators are closed and densely defined on $\ell^2$. Furthermore,  the closed quadratic forms $Q^\circ_D$ and $Q^\circ_N,$ give rise to the corresponding up-, down- and Hodge Laplacians which are self-adjoint operators and satisfy, cf.~\cite[Corollary~3.8]{BK},
$$\Delta^+_{D} = \partial_N\delta_D,\quad\Delta^-_{D} = \delta_N\partial_D\quad\text{and}\quad\Delta^H_{D}= (\delta+\partial)_N(\delta+\partial)_D,$$
    $$\Delta^+_N = \partial_D\delta_N,\quad\Delta^-_N = \delta_D\partial_N\quad\text{and}\quad\Delta^H_N = (\delta+\partial)_D(\delta+\partial)_N.$$

\subsection{Laplacians on graphs}\label{graphs}
Although graphs are $1$-dimensional simplicial complexes and correspondingly their Laplacians are included in the setting above, we  introduce them here separately for specific pedagogical reasons. First of all, we study links of simplicial complexes which are graphs and, therefore, want to distinguish these in parallel studied settings clearly in notation. Secondly, since we aim to apply the theory developed on graphs to the links, it turns out to be more convenient to have the setting for graphs  from \cite{KLW} at hand. So, given a discrete set $X,$ a \emph{graph} over $X$ is a symmetric function $b:X\times X\rightarrow [0,\infty)$ that has a vanishing diagonal and is locally summable, i.e., for all $x\in X$ one has $b(x,x) = 0$ and $$\sum_{y\in X}b(x,y)<\infty.$$ 
We define the quadratic form $\mathcal{Q}$ as 
\[\mathcal{Q}(f) = \frac{1}{2}\sum_{x,y\in X}b(x,y)\vert f(x)- f(y)\vert^2,\] which takes finite values for functions $f$ of \emph{finite energy}, i.e., in the space \[ \mathcal{D}=\lbrace g\in C(X)\mid \sum_{x,y\in X}b(x,y)\vert g(x)-g(y)\vert^2<\infty\rbrace.\]

Furthermore, similar to simplicial complexes, but only on the vertex set, we fix a weight function $m:X\rightarrow(0,\infty)$ and define the \emph{formal graph Laplacian} $\mathcal{L}$ via 
\[\mathcal{L}f(x) = \frac{1}{m(x)}\sum_{y\in X}b(x,y)(f(x)-f(y)),\]
on functions $f$ in the \emph{formal domain} \[ \mathcal{F} = \lbrace g\in C(X)\mid\sum_{y\in X}b(x,y)\vert g(y)\vert<\infty\text{ for all } x\in X\rbrace.\] 

\begin{remark} As mentioned above, graphs are weighted $1$-dimensional simplicial complexes. Given a graph $b$ over $(X,m)$, the edges are given by $\lbrace x,y\rbrace$ with $b(x,y)>0$  and $m$ extends to the edges via $m(\lbrace x,y\rbrace) = b(x,y).$ 
    Furthermore, one then  finds \[\mathcal{D} = \lbrace f\in C(X)\mid\delta f\in\ell^2\rbrace\quad\text{and} \quad\mathcal{Q}(f) = \Vert\delta f\Vert_2^2 = \mathcal{Q}^+(f).\] The graph Laplacian $\mathcal{L}$ acts as $\partial\delta.$
\end{remark}

\section{Localization}

In the section above, we have seen that graphs may naturally be interpreted as simplicial complexes. Here, we find graph structures in a different way within simplicial complexes via links of simplices. Reducing the operators on the simplicial complexes to those links is similar to Garlands famous localization method, cf.~\cite{Garland}. Similar methods were also used in \cite{AharoniBergerMeshulam,ballmann1997l2,GundertWagner,lew2025eigenvalue,oppenheim_local}.\medskip

For a weighted simplicial complex $(\Sigma,m)$, we define the \emph{link} of $\rho\in\Sigma$ as 
$$\mathrm{lk}(\rho) = \lbrace v\in\Sigma_0\mid (v\cup\rho)\succ\rho\rbrace$$
and, for $v\in\lk$, we write $$v\rho=v\cup\rho.$$
On $\lk$, we define $m_\rho\colon\lk\rightarrow(0,\infty)$ and $b_\rho:\lk\times\lk\rightarrow [0,\infty)$ as
$$m_\rho(v) = m(v\rho),\qquad b_\rho(v,v') = m( vv'\rho),\qquad b_\rho(v,v) = 0$$ for all $v,v'\in\lk$ with $v\neq v'.$

By virtue of local summability of $m$, we obtain a graph over $\lk$ in the sense of Subsection \ref{graphs}. Let $\mathcal{L}_\rho$ be the formal Laplacian of the graph $b_\rho$ with formal domain $\mathcal{F}_\rho$ and let $\mathcal{Q}_\rho$ be the associated quadratic form with formal domain $\mathcal{D}_\rho.$ We will write $\ell^p_\rho$ for $\ell^p(\lk,m_\rho)$ and $p\in [1,\infty].$

To relate $\partial$ and $\delta$ to $\mathcal{L}_\rho,$ we define the maps
\[\pi^\rho\colon C(\lk)\rightarrow C(\Sigma),\quad \pi^\rho u(\tau) = 1_{\tau\succ\rho}u(\tau\backslash\rho)\theta(\rho,\tau)\]
and 
\[\pi_\rho\colon C(\Sigma)\rightarrow C(\lk),\quad \pi_\rho\omega(v) =\theta(\rho,v\rho)\omega(v\rho). \]

Notice that, for $u\in C_c(\lk)$ and $\omega\in C_c(\Sigma)$, one has $\pi^\rho u\in C_c(\Sigma)$ and $\pi_\rho\omega\in C_c(\lk).$ Similarly, for $u\in\ell^p_\rho$ and $\omega\in\ell^p$ we have $\pi^\rho u\in\ell^p$ and $\pi_\rho\omega\in\ell^p_\rho.$ 
 If $\omega\in C(\Sigma)$ and $\varphi\in C_c(\lk),$ then \[\sum_{\lk} m_\rho(\pi_\rho\omega)\varphi = \sum_{\Sigma} m\omega(\pi^\rho\varphi).\]
Further, note that \[m_\rho(\lk) = \sum_{\tau\succ\rho}m(\tau)<\infty\] by local summability. Consequently, we always find that the constant function $1$ is in $\ell^p_\rho$ for all $p\in[1,\infty].$

The next lemma captures exactly what we mean by localization. It shows how the boundary and coboundary operator on the simplicial complex are related to the graph Laplacian on the link. Moreover, it relates the quadratic form $\mathcal{Q}_\rho$ to $\mathcal{Q}^+.$ 

\begin{lemma}[Localization lemma] \label{locallemma}Let $\rho\in\Sigma.$
    \begin{itemize}
        \item [(a)] If $\omega\in D(\partial),$ then $\pi_\rho\omega\in \ell^1_\rho $ and \[m(\rho)\partial\omega(\rho) = \sum_{\lk}m_\rho\pi_\rho\omega.\]
        \item[(b)] If $u\in \ell^1_\rho,$ then $\pi^\rho u\in D(\partial)$ and \[\sum_{\lk}m_\rho u = m(\rho)\partial \pi^\rho u(\rho).\]
        \item [(c)] If $u\in\mathcal{D}_\rho,$ then $\delta\pi^\rho u\in\ell^2$ and
        \[\mathcal{Q}_\rho(u) = \mathcal{Q}^+(\pi^\rho u).\]
        \item[(d)] If $u\in\mathcal{D}_\rho$ and $\tau\succ\rho,$ then
        \[\pi^\rho\mathcal{L}_\rho u(\tau) = \partial\delta \pi^\rho u(\tau).\]
        \item[(e)] If $u\in\mathcal{D}_\rho$ and $\mathcal{L}_\rho u\in\ell^1_\rho,$ then $\partial\delta \pi^\rho u\in D(\partial).$
    \end{itemize}
\end{lemma}

\begin{proof} It is straightforward to check    (a) and (b).
    
    To see (c), note that for $\sigma\in\Sigma$ and $\rho\subseteq\sigma$ with $\dim(\rho)+2=\dim(\sigma)$, there are exactly two vertices $v\neq v'$ such that $$v\rho,v'\rho\prec\sigma.$$
    For these $v$ and $v'$,
    one has $\theta(\rho,v\rho)\theta(v\rho,\sigma)+\theta(\rho,v'\rho)\theta(v'\rho,\sigma) =\delta\delta1_\rho(\sigma) =  0.$ Thus,
    \begin{align*}
         \mathcal{Q}^+(\pi^\rho u) &= \sum_{\sigma\in\Sigma}m(\sigma)\Big\vert\sum_{\tau\prec\sigma,\rho\prec\tau}\theta(\tau,\sigma)\theta(\rho\,\tau)u(\tau\backslash\rho)\Big\vert^2 \\&= \frac{1}{2}\sum_{v,v'\in\lk}b_\rho(v,v')\vert u(v)-u(v')\vert^2 = \mathcal{Q}_\rho(u).
    \end{align*}
    To see (d), note first that $u\in\mathcal{F}_\rho$ according to \cite[Proposition~1.4~(b)]{KLW}. Moreover, it follows from (c) that $\delta \pi^\rho u\in\ell^2\subseteq D(\partial).$ Now, by a similar argument as in (c)  
    \begin{align*}
        m(\tau)\partial\delta \pi^\rho u(\tau)&= \sum_{\sigma\succ\tau}m(\sigma)\theta(\tau,\sigma)\sum_{\tau'\prec\sigma}\theta(\tau',\sigma)1_{\rho\prec\tau'}\theta(\rho,\tau')u(\tau'\backslash\rho)\\
        & = \theta(\rho,\tau)\sum_{v\in\lk}b_\rho(\tau\backslash\rho,v)
        (u(\tau\backslash\rho)-u(v))\\
        &= \theta(\rho,\tau)m_\rho(\tau\backslash\rho)\mathcal{L}_\rho u(\tau\backslash\rho) = m(\tau) \pi^\rho\mathcal{L}_\rho u(\tau).
    \end{align*}

In order to see (e), let $\rho'\in\Sigma$ and 
\[\sum_{\tau\succ\rho'} m(\tau)\vert\partial\delta \pi^\rho u\vert(\tau) = \sum_{\tau\succ\rho,\rho'}m(\tau)\vert\pi^\rho\mathcal{L}_\rho u\vert(\tau)+\sum_{\tau\nsucc\rho,\tau\succ\rho'}m(\tau)\vert\partial\delta \pi^\rho u\vert(\tau).\] Here we used (d) for the first sum, which we conclude to be finite as $\mathcal{L}_\rho u\in\ell^1_\rho.$ The second sum is also finite, as the following estimate shows
\begin{align*}
    \sum_{\tau\nsucc\rho,\tau\succ\rho'}m(\tau)\vert\partial\delta \pi^\rho u\vert(\tau)&\leq \sum_{\tau\nsucc\rho,\tau\succ\rho'}\sum_{\sigma\succ\tau}m(\sigma)\vert\delta \pi^\rho u\vert(\sigma)\\
    &\leq \Vert \delta \pi^\rho u\Vert_2\left(\sum_{\tau\nsucc\rho,\tau\succ\rho'}\sum_{\sigma\succ\tau,\sigma\supseteq\rho} m(\sigma)\right)^{1\slash 2}.
\end{align*}
The double sum in the last estimate is finite, since the inner sum runs over a subset of all $\sigma\in\Sigma$ that contain both $\rho$ and $\rho'.$ More precisely, if $\vert\rho\backslash\rho'\vert\geq 2,$ then there can be at most one such $\sigma.$ As $\rho\neq\rho',$ the only other case is $\vert\rho\backslash\rho'\vert = 1$ in which $\sigma\succ\rho\cup\rho'$ and finiteness follows from local summability, as then \[\sum_{\tau\nsucc\rho,\tau\succ\rho'}\sum_{\sigma\succ\tau}m(\sigma)\leq (\dim(\rho)+2)\sum_{\sigma\succ\rho\cup\rho'}m(\sigma).\]
This finishes the proof.
\end{proof}

\section{Characterization of $\partial\partial = 0$}
In this section, we apply the localization technique to characterize when the boundary operator $\partial$ satisfies the complex property $\partial\partial = 0.$ While    $\delta\delta=0$ holds by definition, it was shown in \cite[Example~2.5]{BK} that $\partial\partial = 0$ can fail.
On the other hand, in \cite[Lemma~3.2]{BK} we proved that $\partial\partial = 0$ in $\ell^2,$ given that $m$ satifies for all $\rho\in\Sigma$ the strong local summability condition
\[\sum_{\tau\succ\rho}\sum_{\sigma\succ\tau}m(\sigma)<\infty.\] 
Interpreting this in terms of links, this means that for all $\rho\in\Sigma$ one has that $b_\rho$ has totally summable edge weights, i.e., \[\sum_{v,v'\in\lk}b_\rho(v,v')<\infty.\] Such graphs are clearly recurrent and were studied e.g. in \cite[Section~4.2]{GHKLW}.

Here we give a  characterization of this fact for functions in $\ell^2$ in terms of the recurrence of the graphs  $b_\rho$ arising from the links. We recall that a connected  graph  $b$ over $X$ is called \emph{recurrent} if the associated Markov chain $Y_n$ with transition probability \[\mathbb{P}(Y_n = y\mid Y_{n-1} = x) = \frac{b(x,y)}{\sum_z b(x,z)} \]
for $x,y\in X$ and $n\geq 1$, 
returns almost surely to every vertex, i.e., for all $x,y\in X$
\[\mathbb{P}(Y_k = y\;\text{ for some }\; k>0\mid Y_0 = x) = 1.\]
A connected graph that is not recurrent is called \textit{transient}. The notion of recurrence can be characterized analytically by the fact that the constant function $1$ can be approximated both pointwise and in energy, i.e., in $\mathcal{Q}$, by a sequence of compactly supported functions. For details, we refer to standard text books  \cite{Woe00} or \cite[Section~2.5 and Section~6.6]{KLW}.

The following is the main result of this section and the paper. It shows that $\partial\partial = 0$ on $\ell^2$ is equivalent to the recurrence of all links. In particular, it shows that the failure of $\partial\partial = 0$ on $\ell^2$ is a local phenomenon that can be detected by analyzing the links.

\begin{theorem}\label{MainResult}
   Let $\rho\in\Sigma.$ Then the following are equivalent:
    \begin{itemize}

        \item [(i)]$\partial\partial\omega(\rho) = 0$ for all $\omega\in        D(\partial\partial)\cap\ell^2.$
        \item [(ii)] All connected components of the graph $b_\rho$ over $\lk$ are recurrent.
    \end{itemize}
\end{theorem}
\begin{proof}
    First, let us assume $(ii)$ and let $X_k$, $k\geq1$, be an enumeration of the connected components of $b_\rho.$ Then, according to \cite[Theorem~6.1]{KLW} we find for every $k$ a sequence $\varphi_n^{(k)}\in C_c(\lk)$ such that $0\leq\varphi_n^{(k)}\leq 1$, $\varphi_n^{(k)}\rightarrow 1_{X_k}$ pointwise and $\mathcal{Q}_\rho(\varphi_n^{(k)})\rightarrow0$ as $n\rightarrow\infty.$ We may assume that $\mathcal{Q}_\rho(\varphi_n^{(k)})\leq \frac{1}{n2^k}.$ Let $\omega\in D(\partial\partial)\cap\ell^2$ and let \[\varphi_n = \sum_{k=1}^{n}\varphi_n^{(k)}.\] Then $\varphi_n\in C_c(\lk)$ converges pointwise to the constant function $1$ and \[\mathcal{Q}_\rho(\varphi_n) = \sum_{k=1}^n\mathcal{Q}_\rho(\varphi_n^{(k)})\leq \frac{1}{n}.\] We then use Lemma~\ref{locallemma}~(a), dominated convergence, Stokes' Theorem, \cite[Theorem~2.9]{BK}, and Lemma~\ref{locallemma}~(a) again to obtain the following estimate
    \begin{multline*}
        m(\rho)\vert\partial\partial\omega(\rho)\vert = \lim_{n\rightarrow \infty}\Big\vert\sum_{\lk}m_\rho(\pi_\rho\partial\omega)\varphi_n\Big\vert
        = \lim_{n\rightarrow \infty} \Big\vert\sum_{\Sigma}m(\partial\omega)(\pi^\rho\varphi_n)\Big\vert\\
        = \lim_{n\rightarrow \infty}\Big\vert\sum_\Sigma m\omega(\delta\pi^\rho\varphi_n)\Big\vert
        \leq \lim_{n\rightarrow \infty}\Vert\omega\Vert_2\mathcal{Q}^+(\pi^\rho\varphi_n)^{1\slash 2}.
    \end{multline*}

    From Lemma \ref{locallemma} (b) we conclude that $\partial\partial\omega(\rho) = 0.$\\
    Now, let us assume that $(ii)$ does not hold. Then, according to \cite[Theorem~6.1]{KLW} we find $u\in\mathcal{D}_\rho$ such that $\mathcal{L}_\rho u\in\ell^1_\rho$ and 
    $$0\neq \sum_{\lk}m_\rho\mathcal{L}_\rho u= m(\rho)\partial\pi^\rho\mathcal{L}_\rho u(\rho),$$ 
    where the second equality follows from Lemma~\ref{locallemma}~(b).
    By Lemma~\ref{locallemma} (e) we have $\partial\delta \pi^\rho u\in D(\partial)$ and
    we conclude from Lemma~\ref{locallemma}~(d) that 
        \[0\neq m(\rho)\partial\partial\delta\pi^\rho u(\rho)\]
    which finishes the proof since $\delta\pi^\rho u\in\ell^2$ as $u\in\mathcal{D}_\rho$ by  Lemma~\ref{locallemma}~(c).
\end{proof}
\begin{remark}
    It is natural to ask for conditions when $\partial\partial\omega(\rho) = 0$ for a larger class of functions $\omega.$ In fact, it is not difficult to adapt the argument in Theorem~\ref{MainResult} to show that the recurrence of all connected components of $b_\rho$ implies that $\partial\partial\omega(\rho) = 0$ for all $\omega\in\ell^p$, $p\in[1,2].$ The reason for that is the estimate
    \[\Vert\delta\pi^\rho\varphi_n\Vert_q^q\leq2^{q-2}\mathcal{Q}^+(\pi^\rho\varphi_n)\] for all $2\leq q<\infty.$ If $1<p\leq 2,$ then we apply this to the H\"older dual of $p.$ If $p = 1,$ then dominated convergence does the trick. For $p>2$ this argument does not work. It would be interesting to know if this is related to the existence of an $\ell^p$-Hodge decomposition or boundedness of Riesz-Transforms in $\ell^p.$
\end{remark}

\begin{remark}
    There are various criteria for recurrence of graphs.  For an exhaustive treatment of the probabilistic aspects, one of the standard text books is \cite{Woe00}. For a more analytic approach, we refer to \cite[Chapter~6]{KLW} or \cite{JP} and the references therein. A recent criterion that is particularly relevant for our setting is given in terms of intrinsic metrics on the graph, see \cite[Theorem~4.2]{LPS}. It states that a graph is recurrent if and only if it admits an intrinsic metric with finite metric balls. Furthermore, let us emphasize that the volume of the link with respect to     $m_\rho$  is finite. As a result, recurrence of the links is equivalent to stochastic completeness or form uniqueness, see e.g. \cite[Theorem 6.4]{Schmidt2017} or \cite[Theorem 7.1]{GHKLW}.
\end{remark}

The more subtle question whether $\partial\partial_N = 0$ is not characterized by Theorem~\ref{MainResult}, since in general $\partial\partial_N$ is only a restriction of $\partial\partial\mid_{\ell^2}$. However, we can characterize this under additional assumptions on $(\Sigma,m)$. The relevance of this question will be discussed in the next section. We call $(\Sigma,m)$   \textit{locally balanced} if for all $\rho\in\Sigma$   one has
\[\sup_{ {v,v'\in\lk,v\neq v',w\in\rho} }\frac{m(vv'\rho)}{m(vv'\rho\setminus w)}<\infty,\]
where $\sup\varnothing=0$ in the case of $\rho=\varnothing.$ 
We recall that in   \cite{BK} a global notion of balancedness was introduced, where it is said that $(\Sigma,m)$ is \textit{balanced} if 
\[\sup_{{\sigma,\tau\in\Sigma,}{\tau\prec\sigma}}\frac{m(\sigma)}{m(\tau)}<\infty,\]
which is stronger than the local balancedness. 
The local notion is sufficient for our purposes and is more natural in the context of links.
This notion gives rise to our second main result.

\begin{theorem} 
    Assume that $(\Sigma,m)$ is locally balanced. Then the following are equivalent:
\begin{itemize}
        \item [(i)]$\partial\partial_N = 0$.
        \item [(ii)] All connected components of the graph $b_\rho$ over $\lk$ are recurrent.
\end{itemize}
\end{theorem}
\begin{proof}
    The implication $(ii)\Rightarrow (i)$ follows directly from Theorem~\ref{MainResult}.
    Now assume $(ii)$ is not fulfilled. Then, for some $\rho\in\Sigma,$ the graph $b_\rho$ has a transient connected component.    
    According to \cite[Theorem~6.1~(x)]{KLW} we find a monopole in $\lk.$ More precisely, we find a $v_0\in\lk$ and $u\in\mathcal{D}_\rho$ such that $\mathcal{L}_\rho u = 1_{v_0}.$ Now Lemma~\ref{locallemma} tells us that $\delta\pi^\rho u\in\ell^2$ and $\partial\delta\pi^\rho u\in D(\partial),$ as well as \[\partial\partial\delta\pi^\rho u(\rho) = \partial\pi^\rho\mathcal{L}_\rho u(\rho) = \partial\pi^\rho1_{v_0}(\rho) = \frac{m(v_0\rho)}{m(\rho)}\neq 0.\]
    Let 
    $\omega = \delta\pi^\rho u.$ If $\partial\omega\in\ell^2\subseteq D(\partial),$ then $\omega\in D(\partial\partial_N)$, as well as     
    $\partial\partial_N\omega \neq0.$  We therefore show that this is indeed the case and check that $\partial\omega$ is square-summable over $\tau\succ \rho$ and $\tau\nsucc \rho$ separately.
    
    First, we observe with Lemma~\ref{locallemma}~(d) that 
        \[\sum_{\tau\succ\rho}m(\tau)\vert\partial\omega(\tau)\vert^2  
        = \sum_{\tau\succ\rho}m(\tau)\vert\partial\delta\pi^\rho u(\tau)\vert^2  
        =\sum_{\tau\succ\rho}m(\tau)\vert\pi^\rho\mathcal{L}_\rho u(\tau)\vert^2<\infty.\]
     The finiteness here follows because $\pi^\rho\mathcal{L}_\rho u = \pi^\rho 1_{v_0}\in C_c(\Sigma).$ For the remaining $\tau\nsucc \rho$, we proceed with the following calculation 
\begin{align*}
    \sum_{\tau\nsucc\rho}m(\tau)\vert\partial\omega(\tau)\vert^2 & = \sum_{\tau\nsucc\rho}m(\tau)\vert\partial\delta \pi^\rho u(\tau)\vert^2\\
    & = \sum_{\tau\nsucc\rho}\frac{1}{m(\tau)}\Big\vert\sum_{\sigma\succ\tau}m(\sigma)\theta(\tau,\sigma)\sum_{\tau'\prec\sigma,\tau'\succ\rho}\theta(\tau'\sigma)\theta(\rho,\tau') u(\tau'\backslash\rho)\Big\vert^2\end{align*}
 Note that there is at most one $\sigma=vv'\rho$ that satisfies $\rho\nprec\tau\prec\sigma$ and $\rho\subseteq\sigma,$ namely $\sigma = \rho\cup\tau,$ provided this is an actual element of $\Sigma.$ Furthermore, we denote by $\tau_1=v\rho$ and $\tau_2=v'\rho$ the two unique simplices that satisfy $\rho\prec\tau_1,\tau_2\prec\sigma.$ 
 Using this observation and the identity
\(\theta(\rho,\tau_1)\theta(\tau_1,\sigma)+\theta(\rho,\tau_2)\theta(\tau_2,\sigma) = \delta\delta 1_\rho(\sigma) = 0\),  we obtain
    \begin{align*}
    \ldots& = \sum_{\tau\nsucc\rho}\frac{1}{m(\tau)}\sum_{\sigma\supseteq\rho,\sigma\succ\tau}m^2(\sigma)\vert u(\tau_1\backslash\rho)-u(\tau_2\backslash\rho)\vert^2\\
    & = \frac{1}{2}\sum_{v,v'\in\lk}b_\rho(v,v')\vert u(v)-u(v')\vert^2\sum_{\tau\nsucc\rho,\tau\prec vv'\rho}\frac{m(vv'\rho)}{m(\tau)}\\
    &\leq \mathcal{Q}_\rho(u)(\dim(\rho)+1)
    \sup_{
    \tiny \rho\not\prec\tau\prec vv'\rho}\frac{m(vv'\rho)}{m(\tau)}
    <\infty.
\end{align*}
Finiteness in the last step follows because $u\in\mathcal{D}_\rho$ and local balancedness of $(\Sigma,m)$. We conclude that $\partial\omega\in\ell^2,$ which finishes the proof.
\end{proof}

\section{Applications}
In this section, we give some applications of the results in the previous sections. We first show that under the assumption that all links are recurrent and $\partial\delta C_c(\Sigma)\subseteq\ell^2,$ the Hodge Laplacian $\Delta^H$ associated to any closed quadratic form $Q^H_D\subseteq Q^H\subseteq Q^H_N$ acts as $\partial\delta+\delta\partial.$ Then, we show that under the assumption of recurrent links $\delta_D$ yields an ideal boundary condition on $(\Sigma,m),$ which we use to define relative cohomology classes and prove the respective weak Hodge decomposition. Afterwards, we show that if $ \partial\partial= 0$, then $\Delta^+_D$ has non-trivial kernel. Finally, we discuss the relationship with the transience criterion $(\mathrm{T}')$ of \cite{Parzanchevski}.

\subsection{Action of Hodge Laplacians}

Given any closed quadratic form $Q^H_D\subseteq Q\subseteq Q^H_N,$ we obtain by general theory a unique positive self-adjoint operator $\Delta^H$ such that for all $\omega\in D(\Delta^H)$, $\eta\in D(Q^H)$
\[Q^H(\omega,\eta) = \langle\Delta^H\omega,\eta\rangle.\]
In \cite[Theorem~3.4]{BK} it was shown that on $D(\Delta^H)$ 
\[\Delta^H = (\delta+\partial)^2.\]
It is a natural question whether $\Delta^H$ actually acts as $\delta\partial+\partial\delta$ in analogy to Hodge theory for manifolds. Since $\delta\partial+\partial\delta$ maps $\ell^2(\Sigma_k,m)$ into $C(\Sigma_k),$ this in particular means that $\Delta^H$ acts orthogonally on $\ell^2(\Sigma_k,m)$, $k\geq 0,$ i.e.,
\[\Delta^H = \bigoplus_k\Delta^H_k\] with operators $\Delta^H_k$ defined on $\ell^2(\Sigma_k,m).$ This is always true for $Q^H = Q^H_D$ (\cite[Theorem~3.4]{BK}) but may fail in general, in part due to the fact that $\partial\partial\neq 0.$ 
We address this question in the following theorem.
\begin{theorem}
    Let $\Delta^H$ be the Hodge Laplacian associated to a closed quadratic form $Q^H_D\subseteq Q^H\subseteq Q^H_N.$ Assume that for all $\rho\in \Sigma$ the connected components of $b_\rho$ are recurrent and $\partial\delta C_c(\Sigma)\subseteq\ell^2.$ Then,  on $D(\Delta^H)$ \[\Delta^H = \partial\delta+\delta\partial .\]
\end{theorem}
\begin{proof}
Let $\omega\in D(\Delta^H).$ The first thing to note is that  $(\delta+\partial)\omega\in\ell^2$   as $D(\Delta^H)\subseteq D(Q^H)$ and consequently $(\delta+\partial)\omega\in D(\partial)$ as $\ell^2\subseteq D(\partial)$  by  \cite[Lemma~3.1~(b)]{BK}. According to \cite[Theorem~3.4]{BK} the operator $\Delta^H$ acts as $(\delta+\partial)^2$ and so \[\Delta^H\omega = (\delta+\partial)^2\omega = \delta(\delta+\partial)\omega+\partial(\delta+\partial)\omega = \delta\partial\omega+\partial(\delta+\partial)\omega.\] 
We want to expand the parentheses in the term $\partial(\delta+\partial)\omega.$ However, it is not a priori clear that  $\delta\omega$ or $\partial\omega$ lie in $D( \partial).$ The assumption $\partial\delta C_c(\Sigma)\subseteq\ell^2$  implies  $\ell^2\subseteq D(\partial\delta),$ see \cite[Proposition~3.18 and Lemma~3.15~(b)]{BK}. Thus, $\delta\omega\in D(\partial)$ as $\omega\in D(\Delta^H) \subseteq \ell^2.$ From above we have $(\delta+\partial)\omega\in D(\partial)$ and since $D(\partial)$ is a vector space, we conclude that $\partial\omega\in D(\partial)$. Hence, we infer with $\partial\partial=0$ from Theorem~\ref{MainResult}
 \[\Delta^H\omega = \delta\partial\omega+\partial\delta\omega+\partial\partial\omega = \delta\partial\omega+\partial\delta\omega,\]
 which finishes the proof.
\end{proof}
\begin{remark}
    The condition $\partial\delta C_c(\Sigma)\subseteq\ell^2$ may seem abstract at first, but can be characterized explicitly in terms of $(\Sigma,m).$ More precisely, it holds iff for all $\tau\in\Sigma,$ \[\sum_{\sigma\succ\tau}\sum_{\tau'\prec\sigma,\tau'\neq\tau}\frac{m(\sigma)^2}{m(\tau')}<\infty,\] see \cite[Proposition~3.18]{BK}. It is, for example, satisfied if $(\Sigma,m)$ is balanced.
\end{remark}

\subsection{Hilbert complexes with ideal boundary conditions}
In this section we show that under the assumption that all links are recurrent, the Dirichlet boundary conditions imposed on $\delta$ yield an ideal boundary condition on $(\Sigma,m).$ This is relevant as it allows to apply the general theory of Hilbert complexes to study the $\ell^2$-cohomology of $(\Sigma,m)$ with respect to different boundary conditions. In particular, it allows to prove the Hodge decomposition and the associated Hodge theorem, which generalize the result in \cite[Theorem~3.11]{BK}.\\
To this end, we recall the definition of a Hilbert complex and refer for further discussion  to \cite{BL} and \cite[Section~3.3]{BK}. Given a sequence of mutually orthogonal Hilbert spaces $H_k$, $k\in\mathbb{Z},$ and closed, densely defined operators $d_k\colon  D(d_k)\subseteq H_k\rightarrow H_{k+1}$ with $\mathrm{range}(d_k)\subseteq D(d_{k+1})$ and $d_{k+1}d_{k}= 0,$ a \textit{Hilbert complex} is a chain complex

\[\ldots\xrightarrow{d_{k-2}} D(d_{k-1})\xrightarrow{d_{k-1}}D(d_k)\xrightarrow{d_k}D(d_{k+1})\xrightarrow{d_{k+1}}\ldots.\]
In analogy to elliptic complexes on Riemannian manifolds \cite{BL} we call any closed restriction $\delta_D\subseteq\delta_I\subseteq\delta_N,$ that satisfies \[\mathrm{range}(\delta_{I})\subseteq D(\delta_{I})\]
 an \textit{ideal boundary condition} on $(\Sigma,m).$ That way, we obtain a Hilbert complex with $d_k = \delta_{I,k}$ and $H_k = \ell^2(\Sigma_k,m)$ for all $k\geq 0$ and $H_k = \lbrace 0\rbrace$, $d_k = 0$ when $k<0,$ where $\delta_{I,k}$ is the restriction of $\delta_I$ to $\ell^2(\Sigma_k,m).$\\
 In \cite[Section~3.3]{BK} it was shown that $\delta_I = \delta_N$ is always a valid choice for an ideal boundary condition and yields the (reduced) $\ell^2$-cohomology 
 $$H_{abs}^k(\Sigma) = \mathrm{ker}(\delta_{N,k})\slash\overline{\mathrm{range}(\delta_{N,k-1})}.$$ 
 Here, $\delta_N$ yields the analogue of \textit{absolute boundary conditions} for elliptic complexes as in \cite{BL}. One is also interested in other boundary conditions and their respective cohomologies. In particular, it is natural to ask whether $\delta_D$ yields an ideal boundary condition as well, which may be interpreted as analogue of \textit{relative boundary conditions} on elliptic complexes. However, the crucial condition \[\mathrm{range}(\delta_D)\subseteq D(\delta_D)\] fails to hold in general. It turns out that this is equivalent to $\partial\partial_N = 0.$ Note that in that case one has $\partial\partial_N = \partial_N\partial_N.$

 \begin{theorem}\label{ideal}
     The following assertions are equivalent:
     \begin{itemize}
         \item[(i)] $\delta_D$ is an ideal boundary condition.
         \item[(ii)] $\partial\partial_N = 0$.
     \end{itemize}
      If, for all $\rho\in\Sigma,$ all connected components of $b_\rho$ over $\lk$ are recurrent, then $\delta_D$ is an ideal boundary condition.
 \end{theorem}
\begin{proof}
    Note that $\delta_D$ and $\partial_N$ are adjoints of each other, i.e., $\delta_D = \partial_N^\ast$ and $\partial_N = \delta^\ast_D$, see \cite[Lemma~3.7]{BK}.

   Now assume $(i)$ and let $\omega\in D(\partial_N)$. Then for any $\rho\in\Sigma, $ \[m(\rho)\partial\partial_N\omega(\rho) ) = \langle\partial_N\omega,\delta_D1_\rho\rangle = \langle\omega,\delta_D\delta_D1_\rho\rangle = 0,\] where we used Stokes' theorem, \cite[Theorem~2.9]{BK} for the first equation (recall that $\partial_N\omega\in\ell^2\subseteq D(\partial)$) and $\delta_D1_\rho\in D(\delta_D) = D(\partial_N^*),$ which is due to $(i)$.\\
   On the other hand, if $(ii)$ holds, then for $\eta\in D(\delta_D)$ and $\omega = \delta_D\eta,$ we get for all $f\in D(\partial_N)$
   \[\langle\omega,\partial_Nf\rangle = \langle\delta_D\eta,\partial_N f\rangle = \langle\eta,\partial\partial_N f\rangle = 0 = \langle0,f\rangle.\] Here we used again Stokes' theorem, \cite[Theorem~2.9]{BK} for the second equation, together with the fact that $\lim_n\Vert\eta-\eta_n\Vert_2+\Vert\delta\eta-\delta\eta_n\Vert_2 = 0$ for some sequence $\eta_n\in C_c(\Sigma).$
   In particular $\omega\in D(\partial_N^*) = D(\delta_D),$ which shows $(i)$.

   The second statement  follows directly from the equivalence and Theorem~\ref{MainResult}.
\end{proof}

The theorem allows to define the relative (reduced) cohomology class
$$H^k_{rel}(\Sigma) = \mathrm{ker}(\delta_{D,k})\slash\overline{\mathrm{range}(\delta_{D,k-1})}.$$ 
The terminology ``relative'' comes from the relative boundary conditions. For manifolds, this also means relative to the topological boundary of the manifold. It would be interesting to identify a corresponding analogue of the topological boundary for simplicial complexes. We end this section with the following two corollaries, of which the first is a generalization of \cite[Theorem~3.11]{BK}.

\begin{corollary}[Weak Hodge decomposition and Hodge theorem]
    If for all $\rho\in\Sigma$, all connected components of $b_\rho$ over $\lk$ are recurrent, then for $\Delta= \delta_D\partial_N+\partial_N\delta_D$, 
    	$$\ell^2 = \ker\Delta \oplus\overline{\mathrm{range}(\delta_D)}\oplus\overline{\mathrm{range}(\partial_N)}$$
	and for all $k\geq 0$ and $\Delta_k$ the restriction of $\Delta$ to $\ell^2(\Sigma_k,m),$
	$$\ker(\delta_{D,k})\slash \overline{\mathrm{range}(\delta_{D,k-1})}\cong \ker\Delta_{k}$$ 
\end{corollary}
\begin{proof}
    Under the given assumptions, we have $\mathrm{range}(\partial_N)\subseteq D(\partial_N),$ which follows from Theorem~\ref{ideal}. The proof follows along the lines of  \cite[Theorem~3.11]{BK}.
\end{proof}

\begin{corollary}\label{harmonic}
     If $\rho\in\Sigma$ and all connected components of $b_\rho$ are recurrent, then \[\delta_D1_\rho\in D(\delta_D).\] In particular, $\delta_D1_\rho$ is a harmonic eigenfunction of $\Delta^+_D$ whenever $\rho$ has a coface.
\end{corollary}
\begin{proof} The statement follows directly from Theorem~\ref{MainResult} and Theorem~\ref{ideal}. The fact that $\delta_D1_\rho$ is a harmonic eigenfunction follows from the identity $\Delta^+_D\delta_D = \partial_N\delta_D\delta_D = 0.$ Note that $\delta_D1_\rho\neq0$ iff $\rho$ has a coface.
\end{proof}

\begin{remark} Local summability implies that $\delta C_c(\Sigma)\subseteq\ell^2.$ Thus, one has for all $\rho\in\Sigma$ that $1_\rho\in D(\delta_N\delta_N)$ and it trivially follows that $\delta_N1_\rho$ is a harmonic eigenfunction of $\Delta^+_N,$ provided that $\rho$ has a coface. Moreover, one has $\partial C_c(\Sigma)\subseteq C_c(\Sigma)$ and $\partial\partial = 0$ on $C_c(\Sigma).$ From this, one easily deduces that $\partial1_\rho$ is a harmonic eigenfunction of both $\Delta_N^-$ and $\Delta_D^-,$ provided that $\rho$ has faces in $\Sigma,$ which is true if $\dim(\rho)\geq 1.$ However, for $\Delta^+_D$ this question is non-trivial, precisely because $\delta_D1_\rho$ may not be an element of $D(\delta_D).$ On the level of links one can understand the existence of such a harmonic eigenfunction as follows. Suppose that for some $\rho\in\Sigma$ all connected components of $b_\rho$ are recurrent. Then, by local summability, it follows that $m_\rho(\lk)<\infty.$ In that situation recurrence is equivalent to the fact that $\lambda = 0$ is an eigenvalue of the Dirichlet-Laplacian $L^D_\rho$ associated to the graph $b_\rho$, \cite[Theorem~6.1]{KLW}, with constant eigenfunction $1 = \pi_\rho\delta1_\rho.$
\end{remark}

\subsection{Connection to random walks on simplicial complexes}

It is a conceptual challenge to associate a random walk to a simplicial complex. A structural reason is that the quadratic forms of the Laplacians are in general no Dirichlet forms beyond the case of graphs. In \cite{Parzanchevski}, Parzanchevski and Rosenthal propose various notions what it means for a simplicial complex to be transient. They further show that these notions are not equivalent in general as they are for  the case of graphs. One of these proposed notions is the following property for $d$-dimensional simplicial complexes $\Sigma$:
\begin{itemize}
    \item [($\mathrm T'$)] For every $(d-1)$-dimensional simplex $\sigma\in\Sigma$, there is an $\omega\in\ell^2(\Sigma,m)$ supported in $\Sigma_d$   such that $$\partial  \omega = 1_{\sigma}.$$
\end{itemize}

We relate our result to the property ($\mathrm T'$) in the following way.

\begin{theorem}
    Let $\Sigma$ be a $d$-dimensional simplicial complex. 
    \begin{itemize}
        \item [(a)] If $\mathrm {(T')}$ holds, then the link graph $b_\rho$ of every $(d-2)$-dimensional simplex $\rho$ has a transient connected component.
        \item [(b)] If a connected component of the link graph of some simplex $\rho\in\Sigma$ is transient, then for every vertex $v$ in this component there is an $\omega\in\ell^2(\Sigma,m)$ such that for all $\tau\succ\rho,$
        $$\partial \omega (\tau)= 1_{v\rho}(\tau).$$ 
    \end{itemize}
\end{theorem}
\begin{proof}
  \emph{(a) } If for $\sigma\in\Sigma$, there is an $\omega\in\ell^2(\Sigma,m)$  such that $\partial  \omega = 1_{\sigma},$ then $\omega\in D(\partial\partial)\cap\ell^2$ and for $\tau\prec\sigma$ 
    \[\partial\partial \omega(\tau)= \partial 1_\sigma(\tau) = \frac{m(\sigma)}{m(\tau)}\theta(\tau,\sigma)\neq 0.\]
    Hence, the link of any such $\tau\prec \sigma$ has a transient connected component by Theorem~\ref{MainResult}.

 \emph{(b)} Suppose that there is a $\rho\in\Sigma$ and assume that $b_\rho$ has a transient connected component.  Then \cite[Theorem~6.1]{KLW} guaranties the existence of a \textit{monopole} for each vertex, i.e., for all  $v$ in that component there is a $u\in\mathcal{D}_\rho$ such that \[\mathcal{L}_\rho u = 1_v.\] Then Lemma~\ref{locallemma} tells us that for $\tau\succ\rho,$
\[\partial\delta\pi^\rho u      (\tau)= \pi^\rho\mathcal{L}_\rho u(\tau) = \pi^\rho1_v(\tau) = \theta(\rho,v\rho)1_v(\tau\backslash\rho).\] Therefore, we see that $\omega = \theta(\rho,v\rho)\delta\pi^\rho u$ satisfies 
 $$\partial \omega = 1_{v\rho}$$ on $\lbrace\tau\succ\rho\rbrace$ and Lemma~\ref{locallemma} gives $\omega\in \ell^2$ as $u\in\mathcal{D}_\rho$.
 \end{proof}

\begin{remark}
    Note that 
 \emph{(b)} implies ($\mathrm T'$) only in a  \textit{local} sense.  Specifically, one cannot ensure that $\partial\omega(\tau) = 0$ for $\tau\nsucc\rho.$ 
    From the construction of $\omega$ in the proof, one sees that this cannot be expected as the monopole $u$ is only supported in the link of $\rho$ and, therefore, one does not have any control over other values of $\partial\omega$.
\end{remark}

{\bf Acknowledgements.} The authors appreciate the financial support of the DFG  and MK thanks the IIAS for the hospitality.
\bibliographystyle{abbrv}
\bibliography{literature}

@article{Schmidt2017,
  author    = {Schmidt, Marcel},
  title     = {Global properties of Dirichlet forms on discrete spaces},
  journal   = {Dissertationes Mathematicae (Rozprawy Matematyczne)},
  year      = {2017},
  volume    = {522},
  pages     = {1--43},
  doi       = {10.4303/dm.522}
}

@book {Soardi,
    AUTHOR = {Soardi, Paolo M.},
     TITLE = {Potential theory on infinite networks},
    SERIES = {Lecture Notes in Mathematics},
    VOLUME = {1590},
 PUBLISHER = {Springer-Verlag, Berlin},
      YEAR = {1994},
     PAGES = {viii+187},
      ISBN = {3-540-58448-X},
   MRCLASS = {31C20 (05C90 60J45 94C05)},
  MRNUMBER = {1324344},
MRREVIEWER = {Maretsugu\ Yamasaki},
       DOI = {10.1007/BFb0073995},
       URL = {https://doi.org/10.1007/BFb0073995},
}

@article {LPS,
    AUTHOR = {Lenz, Daniel and Puchert, Simon and Schmidt, Marcel},
     TITLE = {Recurrent and (strongly) resolvable graphs},
   JOURNAL = {J. Math. Pures Appl. (9)},
  FJOURNAL = {Journal de Math\'ematiques Pures et Appliqu\'ees. Neuvi\`eme
              S\'erie},
    VOLUME = {186},
      YEAR = {2024},
     PAGES = {1--30},
      ISSN = {0021-7824,1776-3371},
   MRCLASS = {05C63 (31C05 31C20)},
  MRNUMBER = {4745499},
       DOI = {10.1016/j.matpur.2024.04.002},
       URL = {https://doi.org/10.1016/j.matpur.2024.04.002},
}

@article {Garland,
    AUTHOR = {Garland, Howard},
     TITLE = {{$p$}-adic curvature and the cohomology of discrete subgroups
              of {$p$}-adic groups},
   JOURNAL = {Ann. of Math. (2)},
  FJOURNAL = {Annals of Mathematics. Second Series},
    VOLUME = {97},
      YEAR = {1973},
     PAGES = {375--423},
      ISSN = {0003-486X},
   MRCLASS = {20J05},
  MRNUMBER = {320180},
MRREVIEWER = {M.\ S.\ Raghunathan},
       DOI = {10.2307/1970829},
       URL = {https://doi.org/10.2307/1970829},
}

@article {GHKLW,
    AUTHOR = {Georgakopoulos, Agelos and Haeseler, Sebastian and Keller,
              Matthias and Lenz, Daniel and Wojciechowski, Rados\l aw K.},
     TITLE = {Graphs of finite measure},
   JOURNAL = {J. Math. Pures Appl. (9)},
  FJOURNAL = {Journal de Math\'ematiques Pures et Appliqu\'ees. Neuvi\`eme
              S\'erie},
    VOLUME = {103},
      YEAR = {2015},
    NUMBER = {5},
     PAGES = {1093--1131},
      ISSN = {0021-7824,1776-3371},
   MRCLASS = {31E05 (05C63 47B39)},
  MRNUMBER = {3333051},
MRREVIEWER = {Peter\ I.\ Kogut},
       DOI = {10.1016/j.matpur.2014.10.006},
       URL = {https://doi.org/10.1016/j.matpur.2014.10.006},
}

@book {Woe00,
	AUTHOR = {Woess, Wolfgang},
	TITLE = {Random walks on infinite graphs and groups},
	SERIES = {Cambridge Tracts in Mathematics},
	VOLUME = {138},
	PUBLISHER = {Cambridge University Press, Cambridge},
	YEAR = {2000},
	PAGES = {xii+334},
	ISBN = {0-521-55292-3},
	MRCLASS = {60B15 (60G50 60J10)},
	MRNUMBER = {1743100},
	MRREVIEWER = {Donald\ I.\ Cartwright},
	DOI = {10.1017/CBO9780511470967},
	URL = {https://doi.org/10.1017/CBO9780511470967},
}

@article {AharoniBergerMeshulam,
    AUTHOR = {Aharoni, R. and Berger, E. and Meshulam, R.},
     TITLE = {Eigenvalues and homology of flag complexes and vector
              representations of graphs},
   JOURNAL = {Geom. Funct. Anal.},
  FJOURNAL = {Geometric and Functional Analysis},
    VOLUME = {15},
      YEAR = {2005},
    NUMBER = {3},
     PAGES = {555--566},
      ISSN = {1016-443X,1420-8970},
   MRCLASS = {05C69 (05C65)},
  MRNUMBER = {2221142},
MRREVIEWER = {Nigel\ Martin},
       DOI = {10.1007/s00039-005-0516-9},
       URL = {https://doi.org/10.1007/s00039-005-0516-9},
}

@article{BK,
  title={{O}n {H}odge {L}aplacians on {G}eneral {S}implicial {C}omplexes},
  author={Bartmann, Philipp and Keller, Matthias},
  journal={arXiv preprint arXiv:2508.07761},
  year={2025}
}

@article {BL,
    AUTHOR = {Br\"uning, J. and Lesch, M.},
     TITLE = {Hilbert complexes},
   JOURNAL = {J. Funct. Anal.},
  FJOURNAL = {Journal of Functional Analysis},
    VOLUME = {108},
      YEAR = {1992},
    NUMBER = {1},
     PAGES = {88--132},
      ISSN = {0022-1236,1096-0783},
   MRCLASS = {58G05 (46M20 47A53 58A12 58G20)},
  MRNUMBER = {1174159},
MRREVIEWER = {J\"urgen\ Eichhorn},
       DOI = {10.1016/0022-1236(92)90147-B},
       URL = {https://doi.org/10.1016/0022-1236(92)90147-B},
}

@article {F,
    AUTHOR = {Forman, Robin},
     TITLE = {Bochner's method for cell complexes and combinatorial {R}icci
              curvature},
   JOURNAL = {Discrete Comput. Geom.},
  FJOURNAL = {Discrete \& Computational Geometry. An International Journal
              of Mathematics and Computer Science},
    VOLUME = {29},
      YEAR = {2003},
    NUMBER = {3},
     PAGES = {323--374},
      ISSN = {0179-5376,1432-0444},
   MRCLASS = {52B70 (53C20)},
  MRNUMBER = {1961004},
       DOI = {10.1007/s00454-002-0743-x},
       URL = {https://doi.org/10.1007/s00454-002-0743-x},
}

@article {G,
    AUTHOR = {Gaffney, Matthew P.},
     TITLE = {The harmonic operator for exterior differential forms},
   JOURNAL = {Proc. Nat. Acad. Sci. U.S.A.},
  FJOURNAL = {Proceedings of the National Academy of Sciences of the United
              States of America},
    VOLUME = {37},
      YEAR = {1951},
     PAGES = {48--50},
      ISSN = {0027-8424},
   MRCLASS = {53.0X},
  MRNUMBER = {48138},
MRREVIEWER = {K.\ Kodaira},
       DOI = {10.1073/pnas.37.1.48},
       URL = {https://doi.org/10.1073/pnas.37.1.48},
}

@article {HJ,
    AUTHOR = {Horak, Danijela and Jost, J\"urgen},
     TITLE = {Spectra of combinatorial {L}aplace operators on simplicial
              complexes},
   JOURNAL = {Adv. Math.},
  FJOURNAL = {Advances in Mathematics},
    VOLUME = {244},
      YEAR = {2013},
     PAGES = {303--336},
      ISSN = {0001-8708,1090-2082},
   MRCLASS = {55U10 (18G30 31C20)},
  MRNUMBER = {3077874},
MRREVIEWER = {Paul\ G.\ Goerss},
       DOI = {10.1016/j.aim.2013.05.007},
       URL = {https://doi.org/10.1016/j.aim.2013.05.007},
}

@book {JP,
    AUTHOR = {Jorgensen, Palle E. T. and Pearse, Erin P. J.},
     TITLE = {Operator theory and analysis of infinite networks---theory and
              applications},
    SERIES = {Contemporary Mathematics and Its Applications: Monographs,
              Expositions and Lecture Notes},
    VOLUME = {7},
 PUBLISHER = {World Scientific Publishing Co. Pte. Ltd., Hackensack, NJ},
      YEAR = {[2023] \copyright 2023},
     PAGES = {lv+392},
      ISBN = {[9789811265518]; [9789811265532]},
   MRCLASS = {47-02 (05Cxx 47B25 47B93 90B10 94Axx)},
  MRNUMBER = {4641241},
MRREVIEWER = {Bibhas\ Adhikari},
}

@book {K,
    AUTHOR = {Kato, Tosio},
     TITLE = {Perturbation theory for linear operators},
    SERIES = {Classics in Mathematics},
      NOTE = {Reprint of the 1980 edition},
 PUBLISHER = {Springer-Verlag, Berlin},
      YEAR = {1995},
     PAGES = {xxii+619},
      ISBN = {3-540-58661-X},
   MRCLASS = {47A55 (46-00 47-00)},
  MRNUMBER = {1335452},
}

@book {KLW,
    AUTHOR = {Keller, Matthias and Lenz, Daniel and Wojciechowski, Rados\l
              aw K.},
     TITLE = {Graphs and discrete {D}irichlet spaces},
    SERIES = {Grundlehren der mathematischen Wissenschaften},
    VOLUME = {358},
 PUBLISHER = {Springer, Cham},
      YEAR = {2021},
     PAGES = {xv+668},
      ISBN = {978-3-030-81458-8; 978-3-030-81459-5},
   MRCLASS = {05-01 (05C63 35P05)},
  MRNUMBER = {4383783},
       DOI = {10.1007/978-3-030-81459-5},
       URL = {https://doi.org/10.1007/978-3-030-81459-5},
}

@article{lew2025eigenvalue,
  title={An eigenvalue interlacing approach to {G}arland's method},
  author={Lew, Alan},
  journal={arXiv preprint arXiv:2508.17279},
  year={2025}
}

@article {Parzanchevski,
    AUTHOR = {Parzanchevski, Ori and Rosenthal, Ron},
     TITLE = {Simplicial complexes: spectrum, homology and random walks},
   JOURNAL = {Random Structures Algorithms},
  FJOURNAL = {Random Structures \& Algorithms},
    VOLUME = {50},
      YEAR = {2017},
    NUMBER = {2},
     PAGES = {225--261},
      ISSN = {1042-9832,1098-2418},
   MRCLASS = {05C81 (05C50 05D40 55U10 60D05 60G50)},
  MRNUMBER = {3607124},
MRREVIEWER = {Christian\ Stump},
       DOI = {10.1002/rsa.20657},
       URL = {https://doi.org/10.1002/rsa.20657},
}

@article {RT,
    AUTHOR = {Rosenthal, Ron and Tenenbaum, Lior},
     TITLE = {Simplicial spanning trees in random {S}teiner complexes},
   JOURNAL = {Combinatorica},
  FJOURNAL = {Combinatorica. An International Journal on Combinatorics and
              the Theory of Computing},
    VOLUME = {43},
      YEAR = {2023},
    NUMBER = {3},
     PAGES = {613--650},
      ISSN = {0209-9683,1439-6912},
   MRCLASS = {05E45 (05C05 05C80 60B99 60C05)},
  MRNUMBER = {4630470},
MRREVIEWER = {Junyao\ Pan},
       DOI = {10.1007/s00493-023-00038-3},
       URL = {https://doi.org/10.1007/s00493-023-00038-3},
}

@book {T,
    AUTHOR = {Thaller, Bernd},
     TITLE = {The {D}irac equation},
    SERIES = {Texts and Monographs in Physics},
 PUBLISHER = {Springer-Verlag, Berlin},
      YEAR = {1992},
     PAGES = {xviii+357},
      ISBN = {3-540-54883-1},
   MRCLASS = {81Q10 (35Q40 47N50 58G25)},
  MRNUMBER = {1219537},
MRREVIEWER = {P.\ A.\ Mishnayevskiy},
       DOI = {10.1007/978-3-662-02753-0},
       URL = {https://doi.org/10.1007/978-3-662-02753-0},
}

@article{ballmann1997l2,
  title={On {L}2-cohomology and property ({T}) for automorphism groups of polyhedral cell complexes},
  author={Ballmann, Werner and {\'S}wiatkowski, Jacek},
  journal={Geometric and Functional Analysis},
  volume={7},
  pages={615--645},
  year={1997},
  publisher={Springer}
}

@article {oppenheim_local,
    AUTHOR = {Oppenheim, Izhar},
     TITLE = {Local spectral expansion approach to high dimensional
              expanders {P}art {I}: {D}escent of spectral gaps},
   JOURNAL = {Discrete Comput. Geom.},
  FJOURNAL = {Discrete \& Computational Geometry. An International Journal
              of Mathematics and Computer Science},
    VOLUME = {59},
      YEAR = {2018},
    NUMBER = {2},
     PAGES = {293--330},
      ISSN = {0179-5376,1432-0444},
   MRCLASS = {05E45 (05A20)},
  MRNUMBER = {3755725},
MRREVIEWER = {Xiaogang\ Liu},
       DOI = {10.1007/s00454-017-9948-x},
       URL = {https://doi.org/10.1007/s00454-017-9948-x},
}

@article {GundertWagner,
    AUTHOR = {Gundert, Anna and Wagner, Uli},
     TITLE = {On eigenvalues of random complexes},
   JOURNAL = {Israel J. Math.},
  FJOURNAL = {Israel Journal of Mathematics},
    VOLUME = {216},
      YEAR = {2016},
    NUMBER = {2},
     PAGES = {545--582},
      ISSN = {0021-2172,1565-8511},
   MRCLASS = {55U10 (05C80 05E45 58J50 60D05)},
  MRNUMBER = {3557457},
MRREVIEWER = {Dirk\ Sch\"utz},
       DOI = {10.1007/s11856-016-1419-1},
       URL = {https://doi.org/10.1007/s11856-016-1419-1},
}

\end{document}